\documentclass[12pt, reqno]{amsart}       
\usepackage{graphicx}

\usepackage{amsmath}

\usepackage{amssymb}
\newtheorem{thm}{Theorem}

\newtheorem{prop}[thm]{Proposition}
\newtheorem{cor}[thm]{Corollary}
\theoremstyle{definition}
\newtheorem{defn}{Definition}

\newcommand{\vertiii}[1]{{\left\vert\kern-0.25ex\left\vert\kern-0.25ex\left\vert
#1 \right\vert\kern-0.25ex\right\vert\kern-0.25ex\right\vert}}
\def \rk   {\text {\rm rk}}

\def \In   {\text {\rm In}}
\def \diag  {\text {\rm diag}}
\def \Hrk   {\text {\rm Hrk}}
\def \k   {\text {\rm Krk}}

\def \min   {\text {\rm min}}

\def \diag  {\text {\rm diag}}
\def \lim   {\text {\rm lim}}
\def \d   {\text {\rm d}}

\def \wt    {\widetilde}

\def \qed   {\hfill \vrule height6pt width 6pt depth 0pt}

\begin{document}

\title[]{Hadamard powers of rank two, doubly nonnegative matrices}

\author{Tanvi Jain}

\address{Indian Statistical Institute,\\
 New Delhi 110016,\\
 India}
\email{tanvi@isid.ac.in}




\begin{abstract}
We study ranks of the $r\textrm{th}$ Hadamard powers of doubly nonnegative matrices
and show that the matrix $A^{\circ r}$ is positive definite for every $n\times n$ doubly nonnegative matrix $A$
and for every $r>n-2$ if and only if no column of $A$ is a scalar multiple of any other column of $A.$
A particular emphasis is given to the study of rank, positivity and monotonicity of Hadamard powers of
rank two, positive semidefinite matrices that have all entries positive.
\end{abstract}


\maketitle

\section{Introduction}

Let $A=\begin{bmatrix}a_{ij}\end{bmatrix}$ and $B=\begin{bmatrix}b_{ij}\end{bmatrix}$ be two $n\times n$ matrices.
The {\it Hadamard product} of $A$ and $B$ is the matrix
$$A\circ B=\begin{bmatrix}a_{ij}b_{ij}\end{bmatrix}.$$
A classical theorem of Schur, known as the {\it Schur product theorem},  tells us that
the Hadamard product of two positive semidefinite matrices is again positive semidefinite, \cite[Theorem 7.5.3 p.479]{hj}.
The Hadamard product is positive definite if one matrix is positive definite and the other is positive semidefinite with all its diagonal entries positive.
A direct consequence of Schur's product theorem is that if $A$ is a positive semidefinite matrix,
then its $m\textrm{th}$ Hadamard power $A^{\circ m}=\begin{bmatrix}a_{ij}^m\end{bmatrix}$ is positive semidefinite for all nonnegative integers $m.$
And if $A$ is positive definite, then $A^{\circ m}$ is also positive definite for all positive integers $m.$
It is obvious that if a column of $A$ is a scalar multiple of any other column of $A,$
then $A^{\circ m}$ is not positive definite for any $m\ge 0.$
Recently, in \cite{hy}, R. Horn and Z. Yang gave a lower bound for the rank of the Hadamard product of two
positive semidefinite matrices
in terms of their rank and Kruskal rank.
They showed that if $A$ is an $n\times n$ positive semidefinite matrix with no column of $A$ a scalar multiple of
any other column of $A,$
then $A^{\circ m}$ is positive definite for all integers $m>n-2.$

It is natural to ask what happens if the integral Hadamard powers are replaced with the fractional Hadamard powers $A^{\circ r}.$
An $n\times n$ matrix is called {\it doubly nonnegative} if it is positive semidefinite and all its entries are nonnegative.
For a doubly nonnegative matrix $A=\begin{bmatrix}a_{ij}\end{bmatrix}$ and a nonnegative real number $r,$
$A^{\circ r}=\begin{bmatrix}a_{ij}^r\end{bmatrix}$ is the $r\textrm{th}$ fractional Hadamard power of $A.$
There has been much interest in studying the conditions under which
these Hadamard power functions preserve properties such as
positivity, monotonicity, and convexity for various classes of doubly nonnegative matrices.
See \cite{fh,h,ho,gkr,gkr1,gkr2,j}.
If $A$ is a $2\times 2$ doubly nonnegative matrix, then $A^{\circ r}$ is positive semidefinite for all $r\ge 0.$
But this is not true for an $n\times n$ doubly nonnegative matrix when $n\ge 3.$
C. FitzGerald and R. Horn showed in \cite {fh} that $A^{\circ r}$
is positive semidefinite for all $n\times n$ doubly nonnegative matrices $A$ if and only if $r>n-2$ or $r$ is a nonnegative integer.
To prove the necessity, they showed that if $r<n-2$ is not an integer,
then the $n\times n$ matrix $\begin{bmatrix}(1+\epsilon ij)^r\end{bmatrix}$ is not positive semidefinite for sufficiently small $\epsilon>0.$
In \cite{j}, the author extended this result and showed that for any $n$ distinct positive numbers $x_1,\ldots,x_n,$
the $n\times n$ matrix $\begin{bmatrix}(1+x_ix_j)^r\end{bmatrix}$ is positive semidefinite if and only if $r> n-2$ or $r=0,1,\ldots,n-2.$
In this paper we further extend this result to all rank two, positive semidefinite matrices that have all entries positive.
We also study the ranks of Hadamard powers of doubly nonnegative matrices.
In particular, we introduce a concept of Hadamard power rank of matrices,
and show that the result (Corollary 9) given in \cite{hy} for the ranks of integral Hadamard powers
also holds true for the fractional Hadamard powers of doubly nonnegative matrices.

We also address the question of monotonicity of Hadamard powers of doubly nonnegative matrices.
For two Hermitian matrices $A$ and $B,$
we say that $A\ge B$ if $A-B$ is positive semidefinite.
Similar to positivity, we know by Schur's product theorem that if $A$ and $B$ are
positive semidefinite matrices such that $A\ge B,$
then $A^{\circ m}\ge B^{\circ m}$ for all nonnegative integers $m.$
It was shown in \cite{fh} that $A^{\circ r}\ge B^{\circ r}$ for all $n\times n$ doubly nonnegative matrices $A\ge B$
if and only if $r>n-1$ or $r=0,1,\ldots,n-1.$
If $r<n-1$ is not an integer, then there exists a small enough $\epsilon>0$ such that
$\begin{bmatrix}(1+\epsilon ij)^r\end{bmatrix}\ngeq E.$
Here $E$ is the $n\times n$ matrix whose entries are all one.
We show that the same result remains true if the matrices
$\begin{bmatrix} 1+\epsilon ij\end{bmatrix}$ and $E$ are respectively
replaced with a rank two, positive semidefinite matrix and a rank one positive semidefinite matrix
satisfying certain conditions.

\section{Rank and positivity}

Let $A$ be an $n\times n$ matrix with nonnegative entries.
If a column of $A$ is a scalar multiple
of some other column of $A,$
then the same is true for $A^{\circ r}$ for all $r>0.$
Hence, a column of $A$ being a scalar multiple of another column is a sufficient condition for the singularity of $A^{\circ r}$ for all $r>0.$
In this section, we show that this is also a necessary condition for all
$n\times n$ doubly nonnegative matrices and for all $r>n-2.$
We start with a definition that helps us describe the rank of Hadamard powers of doubly nonnegative matrices.

\begin{defn}
Let $A$ be an $m\times n$ nonzero matrix with complex entries.
Let $a_1,\ldots,a_n$ be the columns of $A.$
The {\it Hadamard power rank} of $A$ is the largest integer $k$
for which there is a list $S=\{a_{i_1},\ldots,a_{i_k}\}$ of $k$ distinct columns of $A$ such that no element of $S$ is a scalar multiple of any other element of $S.$
The Hadamard power rank of the zero matrix is zero.
\end{defn}

A related concept is that of {\it Kruskal rank} of a matrix.
See \cite{hy,kr,s,y}.
An $m\times n$ matrix $A$ has Kruskal rank $k$ if $k$ is the largest integer for which every list of $k$ distinct columns of $A$ is linearly independent.
The Kruskal rank of $A$ is zero if any of its columns is zero.
We denote the rank, the Hadamard power rank, and the Kruskal rank of $A$ by
$\rk(A),$ $\Hrk(A),$ and $\k(A),$ respectively.
Clearly for an $n\times n$ matrix $A,$
$$0\le \k(A)\le \rk(A)\le \Hrk(A)\le n.$$
We have $\rk(A)=1$ if and only if $\Hrk(A)=1,$
and 
\begin{equation}
\Hrk(A)=n\Leftrightarrow \k(A)\ge 2.\label{eqrr2}
\end{equation}
An $n\times n$  rank two matrix can have Hadamard power rank $n.$
For example, if
$$A=\begin{bmatrix}1 & 1 & 1 & 1\\
1 & 2 & 3 & 4\\
1 & 3 & 5 & 7\\
1 & 4 & 7 & 10\end{bmatrix},$$
then $\rk(A)=2$ but $\Hrk(A)=4.$
The Hadamard power rank of a matrix remains unchanged under
\begin{itemize}
\item[(i)] any permutation of rows or columns, and
\item[(ii)] left or right multiplication by a nonsingular diagonal matrix.
\end{itemize}
Observe that for any $m\times n$ matrix $A,$ we have $\Hrk(A^*A)\le \Hrk(A).$
Moreover, if $\rk(A)=m,$ then $\Hrk(A^*A)=\Hrk(A).$

We say that a set $S\subseteq \mathbb{R}^n$ is $m$-quasi linearly independent ($m$-qli)if every subset of $S$ with cardinality $m$ is linearly independent.
The set $S$ is quasi linearly independent if it is $n$-qli.
See \cite{fs}.
 The set $S$ is $2$-qli if no element of $S$ is a scalar multiple of some other element of $S.$
Let $A$ be an $n\times n$ matrix and $a_1,\ldots,a_n$ be its columns.
Let $S=\{a_1,\ldots,a_n\}.$
Then $A$ has Hadamard power rank $k$ if and only if there is a $2$-qli subset $\{a_{i_1},\ldots,a_{i_k}\},$ $a_{i_j}\ne a_{i_l},$ $1\le j\ne l\le k,$ of $S$ consisting of $k$ elements,
and no subset of $S$ containing $k+1$ elements is $2$-qli.

The next theorem gives a relationship between the Hadamard power rank of a doubly nonnegative matrix and the rank of its Hadamard power.

\begin{thm}\label{thm_nonsing}
Let $A$ be an $n\times n$ nonzero doubly nonnegative matrix and let $r$ be a positive real number.
Then for all $r>\Hrk(A)-2$ the rank of $A^{\circ r}$ is $\Hrk(A).$
\end{thm}

\begin{proof}
We first prove the result for the special case $\Hrk(A)=n.$
We prove this by induction on $n.$
The statement is trivially true when $n=1.$
Assume that the result holds for $(n-1)\times (n-1)$ doubly nonnegative matrices with Hadamard power rank $n-1.$
Let $A=\begin{bmatrix}a_{ij}\end{bmatrix}$ be an $n\times n$ doubly nonnegative matrix with $\Hrk(A)=n.$
Since $\Hrk(A)=n,$ all columns of $A$ are nonzero,
and since $A$ is positive semidefinite, $a_{nn}\ne 0.$
Let $\zeta$ be the vector $\frac{1}{\sqrt{a_{nn}}}(a_{1n},\ldots,a_{nn})^T,$
and let $C$ be the matrix $C=\zeta\zeta^T.$
Clearly $C$ is a doubly nonnegative matrix with rank one.
As in the proof of Theorem 2.2 of \cite{fh}, we have
\begin{equation}
A^{\circ r}= C^{\circ r}+r\int\limits_{0}^{1}(A-C)\circ \left(tA+(1-t)C\right)^{\circ (r-1)}\d t.\label{eq1nonsing}
\end{equation}
For every $t\in [0,1],$ let $A(t)$ be the $n\times n$ matrix $tA+(1-t)C,$
and let $A_0(t)$ be its leading $(n-1)\times (n-1)$ principal submatrix.
Let $h(t)$ denote the Hadamard power rank of $A_0(t).$

We claim that there exists a $\delta\in (0,1)$
such that $h(t)=n-1$ for all $t\in (\delta,1].$
If such a $\delta$ does not exist,
then there is a sequence $(t_k)$ of distinct terms that converges to $1$ and is such that
$h(t_k)<n-1.$
This means that for every $k$
$A_0(t_k)$ has two linearly dependent columns.
By passing to a subsequence, if necessary,
we can assume that the $i\textrm{th}$ and the $j\textrm{th}$ columns of $A_0(t_k)$ are linearly dependent for all $k.$
Since the sequence $\left(A_0(t_k)\right)$ converges to $A_0(1),$
the $i\textrm{th}$ and $j\textrm{th}$ columns of $A_0(1)$ are also linearly dependent.
The matrix $A_0(1)$ is the leading $(n-1)\times (n-1)$ principal submatrix of $A$
and $A$ is positive semidefinite.
So, the $i\textrm{th}$ and the $j\textrm{th}$ columns of $A$ are linearly dependent by the row inclusion property of positive semidefinite matrices \cite[Observation 7.1.10, p. 432]{hj}.
This contradicts the fact that $\Hrk(A)=n.$
Hence the claim is established.
By induction hypotheses $\left(A_0(t)\right)^{\circ (r-1)}$ is positive definite for all $t\in (\delta,1]$ and $r> n-3.$

Again since $\Hrk(A)=n,$ by the relation \eqref{eqrr2} and Corollary 3 of \cite{hy},
all the $2\times 2$ principal minors of $A$ are nonzero.
This in turn gives that all the diagonal entries of $A-C$ are nonzero.
By Theorem 7.5.3 of \cite{hj}, the leading $(n-1)\times (n-1)$ principal submatrix of $(A-C)\circ \left(A(t)\right)^{\circ (r-1)}$ is positive definite
for all $t\in (\delta,1]$ and $r> n-3.$
Let
$$B=r\int\limits_{0}^{1}(A-C)\circ\left(A(t)\right)^{\circ (r-1)}\d t.$$
The last column of $B$ is zero and $B$ has rank $n-1.$
It can be verified that $C^{\circ r}+B$ has rank $n.$
Therefore, $A^{\circ r}$ is positive definite.

The result is trivial if $\Hrk(A)=1.$
So, let $\Hrk(A)=k\ge 2.$
By a suitable permutation similarity, we can assume that no two of the first $k$ columns of $A$ are linearly dependent.
Then by the row inclusion property of positive semidefinite matrices \cite[Observation 7.1.10, p. 432]{hy}, we know that the leading $k\times k$ principal submatrix $A_0$ of $A$ has Hadamard power rank $k.$
Thus $A_0^{\circ r}$ is positive definite for all $r> k-2.$
This implies that $A^{\circ r}$ has rank $k$ for all $r> k-2.$
\end{proof}

A stronger result is true if $A$ is a rank two, positive semidefinite matrix with all entries positive.
The {\it inertia} of a Hermitian matrix $A$ is the triple $\In(A)=(n_+(A),n_0(A),n_-(A)),$
where $n_+(A),$ $n_0(A)$ and $n_-(A)$ denote the numbers of positive, zero and negative eigenvalues of $A.$

\begin{thm}\label{thm_main}
Let $A$ be an $n\times n$ rank two, positive semidefinite matrix with all its entries positive.
Suppose $A$ has Hadamard power rank $k.$ Then
\begin{itemize}
\item[$(i)$] $A^{\circ r}$ is positive semidefinite with rank $k$ for every positive $r>k-2,$ and therefore $\In\, A^{\circ r}=(k,n-k,0).$
\item[$(ii)$] If $r=0,1,\ldots,k-2,$ then $A^{\circ r}$ is positive semidefinite with rank $r+1$

 and therefore $\In\, A^{\circ r}=(r+1,n-r-1,0).$
\item[$(iii)$] Let $0\le m<r<m+1\le k-2$ for some integer $m.$ Then 
\begin{equation*}
\In\, A^{\circ r}=\bigl(\left\lfloor\frac{k+m+2}{2}\right\rfloor, n-k, \left\lceil \frac{k-m-2}{2}\right\rceil\bigr).
\end{equation*}
\item[$(iv)$] Every nonzero eigenvalue of $A^{\circ r}$ is simple.
\end{itemize}
\end{thm}

\begin{cor}\label{thm_hrk}
Let $A$ be an $n\times n$ doubly nonnegative matrix.
Then
$$\Hrk(A)=\max\{\rk\left(A^{\circ r}\right):r\ge 0\}.$$
In addition, if all the entries of $A$ are positive and $\rk(A)=2,$ then
$\Hrk(A)=\rk \left(A^{\circ r}\right),$ for all positive real numbers $r$ with $r\ne 0,1,\ldots,n-2.$
\end{cor}

The last statement of Corollary \ref{thm_hrk} is not true if we drop the condition that $A$ has rank two.
For example, let $A$ be the $4\times 4$ positive definite matrix
$$A=\begin{bmatrix}1 & 1 & 1 & 1\\
1 & 8 & 27 & 64\\
1 & 27 & 125 & 343\\
1 & 64 & 343 & 1000\end{bmatrix}.$$
Then $\rk(A^{\circ (1/3)})=2$ and $\rk(A^{\circ (2/3)})=3.$

Let $x_1,\ldots,x_n$ be real numbers such that $1+x_ix_j>0$ for all $i,j=1,\ldots,n.$
We will first prove Theorem \ref{thm_main} for a special class of rank two, positive semidefinite, entrywise positive matrices of the form
\begin{equation}
X=\begin{bmatrix}1+x_ix_j\end{bmatrix},\label{eq11}
\end{equation} and then obtain the theorem as a consequence.
Observe that the $j\textrm{th}$ and the $l\textrm{th}$ columns of $X$ are linearly dependent if and only if $x_j=x_l.$
Hence the Hadamard power rank of $X$ is $k$ if and only if exactly $k$ of $x_1,\ldots,x_n$ are distinct.
In particular, $\Hrk(X)=n$ if and only if all $x_1,\ldots,x_n$ are distinct.

The next two results use ideas similar to those in \cite{bj} and \cite{j}
for the matrix of the form $\begin{bmatrix}p_i+q_j\end{bmatrix}$ for distinct positive numbers $p_1,\ldots,p_n$ and $q_1,\ldots,q_n.$
Let $(a_1,\ldots,a_n)$ be an $n$-tuple of nonzero real numbers.
The number of sign changes in $(a_1,\ldots,a_n)$ is the number of indices $m$ for which $a_{m-1}a_m<0,$ $1<m\le n.$
The number of sign changes in any real tuple is the number of sign changes in the tuple obtained by deleting the zero terms and keeping the remaining terms unaltered with their order preserved.
The zero tuple has zero number of sign changes.
See \cite[Part V]{gg}.
\begin{prop}\label{prop1}
Let $x_1<\cdots< x_n$ be distinct real numbers.
Let
$$R_1=\begin{cases}
-\frac{1}{x_n} & x_n>0\\
-\infty & otherwise,
\end{cases}$$
and
$$R_2=\begin{cases}
-\frac{1}{x_1} & x_1<0\\
\infty & otherwise.
\end{cases}$$
For each real number $r,$ and
for every nonzero tuple $(c_1,\ldots,c_n),$
the function
$f_r:(R_1,R_2)\to\mathbb{R}$ defined as
\begin{equation}
f_r(x)=\sum\limits_{j=1}^{n}c_j(1+xx_j)^r.\label{eqprop11}
\end{equation}
has at most $n-1$ zeros counting multiplicities.
(Here we use the convention that the number of zeros of the identically zero function is zero.)
\end{prop}

\begin{proof}
Let $Z(g)$ denote the number of zeros of a function $g,$
and let $V(c_1,\ldots,c_n)$ denote the number of sign changes in the tuple $(c_1,\ldots,c_n).$
We prove that for every real number $r$
and for every nonzero tuple $(c_1,\ldots,c_n),$
\begin{equation}
Z(f_r)\le V(c_1,\ldots,c_n).\label{eqprop12}
\end{equation}
We prove this by induction on $V(c_1,\ldots,c_n).$

Clearly if $V(c_1,\ldots,c_n)=0,$	then $Z(f_r)=0.$
Assume that \eqref{eqprop12} holds for every real $r$
and for all nonzero tuples $(c_1,\ldots,c_n)$ with $V(c_1,\ldots,c_n)$\\
$\le k-1.$
Now let $V(c_1,\ldots,c_n)=k>0.$
Without loss of generality, we can assume that each $c_i\ne 0.$
Choose $1<j\le n,$ such that
$c_{j-1}c_j<0.$
By replacing each $x_i$ by $-x_i$ if necessary,
we can assume that $x_j>0.$
So, we can find a $u>0$ such that $1-x_ju<0$ and $1-x_{j-1}u>0.$
Consider the function $\varphi$ on $(R_1,R_2)$ defined as
$$\varphi(x)=\sum\limits_{i=1}^{n}c_i(1-x_iu)(1+xx_i)^{r-1}.$$
We have $V(c_1(1-ux_1),\ldots, c_n(1-ux_n))=k-1.$
Hence by the induction hypotheses,
$Z(\varphi)\le k-1.$
For all $x$ in $(R_1,R_2),$
we see that
\begin{eqnarray*}
\varphi(x) &=& \sum\limits_{i=1}^{n}c_i(1+xx_i-x_i(x+u))(1+xx_i)^{r-1}\\
&=&\sum\limits_{i=1}^{n}c_i(1+xx_i)^r-\sum\limits_{i=1}^{n}c_ix_i(x+u)(1+xx_i)^{r-1}\\
&=&-\frac{(x+u)^{r+1}}{r}\left(\frac{-r}{(x+u)^{r+1}}f_r(x)+\frac{1}{(x+u)^r}f_r^\prime(x)\right)\\
&=& -\frac{(x+u)^{r+1}}{r}h^\prime(x),
\end{eqnarray*}
where $h(x)=(x+u)^{-r}f_r(x).$
Since $x_j>0,$ we can see that $u>0$ and $R_1+u=-1/x_n+u>-1/x_j+u>0.$
Hence $x+u\ne 0$ for all $x\in (R_1,R_2).$
So, $h$ is a well defined function on $(R_1,R_2)$ and $Z(h)=Z(f_r).$
By Rolle's theorem, $Z(h)\le Z(h^\prime)+1.$
Thus we have $Z(f_r)=Z(h)\le Z(h^\prime)+1=Z(\varphi)+1\le k.$
\end{proof}

\begin{cor}\label{cor2}
Let $x_1<\cdots<x_n$ be distinct real numbers.
Suppose $y_1,\ldots,y_n$ are distinct real numbers such that the $n\times n$ matrix 
\begin{equation}
S=\begin{bmatrix}1+y_ix_j\end{bmatrix}\label{eq42}
\end{equation}
has all its entries positive.
Let $r$ be a real number.
Then $S^{\circ r}$ is nonsingular if $r\ne 0,1,\ldots,n-2.$
If $r=0,1,\ldots,n-2,$ then $S^{\circ r}$ has rank $r+1.$
\end{cor}

\begin{proof}
Let $r\ne 0,1,\ldots,n-2.$
The matrix $S^{\circ r}$ is singular if and only if there exists a nonzero tuple $(c_1,\ldots,c_n)$ satisfying
$$f_r(y_i)=\sum\limits_{j=1}^{n}c_j(1+y_ix_j)^r=0\textrm{  for all  }i=1,\ldots,n.$$
Since each $1+y_ix_j>0,$ each $y_i\in (R_1,R_2).$
By Proposition \ref{prop1} this is possible only when $f_r$ is identically zero.
The function $f_r$ is identically zero if and only if $f_r^{(k)}(y)=0$ for all nonnegative integers $k$ and all $y\in (R_1,R_2).$
For $r\ne 0,1,\ldots,n-2,$
 $f_r^{(k)}(0)=0,$ $0\le k\le n-1$ implies
\begin{equation}
\begin{bmatrix}1 & 1 & \cdots & 1\\
x_1 & x_2 & \cdots & x_n\\
\vdots & \vdots & \vdots\vdots\vdots & \vdots\\
x_1^{n-1} & x_2^{n-1} & \cdots & x_n^{n-1}\end{bmatrix}\begin{bmatrix}c_1\\
c_2\\
\vdots\\
c_n\end{bmatrix}=\begin{bmatrix}0\\
0\\
\vdots\\
0\end{bmatrix}.\label{eqrr1}
\end{equation}
Since the $n\times n$ matrix on the left hand side of \eqref{eqrr1} is an invertible Vandermonde matrix,
we get $c_1=\cdots=c_n=0.$
But this is a contradiction to the choice of $(c_1,\ldots,c_n)$ to be nonzero.
This gives that the function $f_r$ is not identically zero if $r\ne 0,1,\ldots,n-2.$
Hence $S^{\circ r}$ must be nonsingular.

If $r=0,1,\ldots,n-2,$
then $$S^{\circ r}=W_y^TDW_x,$$
  where $D$ is the $(r+1)\times (r+1)$ positive diagonal matrix with $j\textrm{th}$ diagonal entry $\binom{r}{j-1},$
	and $W_x$ is the $(r+1)\times n$ rectangular Vandermonde matrix
	$$W_x=\begin{bmatrix}1 & 1 & \cdots & 1\\
	x_1 & x_2 & \cdots & x_n\\
	\vdots & \vdots & \vdots\vdots\vdots & \vdots\\
	x_1^r & x_2^r & \cdots & x_n^r\end{bmatrix}.$$
	The matrix $W_y$ is the matrix $W_x$ with $x_i$'s replaced by $y_i$'s.
	From this factorisation it is clear that $S^{\circ r}$ has rank $r+1.$
	\end{proof}
	
	\begin{thm}\label{thm3}
	Let $x_1<\cdots<x_n$ be distinct real numbers
	such that all the entries of the $n\times n$ matrix $X$ defined by \eqref{eq11} are positive.
	Let $r$ be a nonnegative real number.
\begin{itemize}
\item[$(i)$] $X^{\circ r}$ is positive definite for every $r>n-2,$ and therefore $\In\, X^{\circ r}=(n,0,0).$
\item[$(ii)$] If $r=0,1,\ldots,n-2,$ then $X^{\circ r}$ is positive semidefinite with rank $r+1,$ and therefore $\In\, X^{\circ r}=(r+1,n-(r+1),0).$
\item[$(iii)$] Let $0\le m<r<m+1\le n-2$ for some integer $m.$ Then 
\begin{equation*}
\In\, X^{\circ r}=\bigl(\left\lfloor\frac{n+m+2}{2}\right\rfloor, 0, \left\lceil \frac{n-m-2}{2}\right\rceil\bigr).
\end{equation*}
\item[$(iv)$] Every nonzero eigenvalue of $X^{\circ r}$ is simple.
\end{itemize}
\end{thm}

\begin{proof}
If the $x_i$'s are all positive, the result is the same as Theorem 2.8 of \cite{j}.
So let $x_i$ be distinct real numbers such that $X$ has all entries positive.
Part $(ii)$ follows from the Schur product theorem and Corollary \ref{cor2}.
For $0\le t\le 1,$ consider the numbers
$$x_i(t)=(1-t)i+tx_i,\ \ 1\le i\le n.$$
Let $X(t)$ be the matrix given by \eqref{eq11} corresponding to the numbers $x_1(t)<\cdots<x_n(t).$
Clearly $X(0)=\begin{bmatrix}1+ij\end{bmatrix}$ and $X(1)=X.$
Since $(i)$ and $(iii)$ hold for $X(0),$
we can show that $(i)$ and $(iii)$ also hold for $X$ by using the nonsingularity of $X(t)$ and the continuity of their eigenvalues.
Part $(iv)$ follows from the sign regularity of $X^{\circ r}.$
The proof is exactly the same as that of Theorem 2.8$(iv)$ of \cite{j}.
\end{proof}

Let $x_1<\cdots<x_k$ be $k$ distinct real numbers and let $n_1,\ldots,n_k$ be positive integers such that $n_1+\cdots+n_k=n.$
Let $X$ be the $n\times n$ matrix given by \eqref{eq11} corresponding to the numbers
${\underset{n_1\, \textrm{times}}{\underbrace{x_1,\ldots,x_1}}},\ldots,{\underset{n_k\, \textrm{times}}{\underbrace{x_k,\ldots,x_k}}}.$
It is easy to see that $\Hrk(X)=k.$
By using Theorem \ref{thm3} and the Cauchy interlacing principle \cite[Ch.III]{rbh},
we can prove the following theorem.

\begin{thm}\label{thm4}
Let $x_1<\cdots<x_k$ be $k$ distinct real numbers, and let $X$ be the $n\times n$ matrix defined in the preceding paragraph.
Suppose that all the entries of $X$ are positive. Let $r$ be a nonnegative real number.
Then $(i)$-$(iv)$ of Theorem \ref{thm_main} hold true for $X.$
\end{thm}

{\noindent {\it{\bf{Proof of Theorem \ref{thm_main}}}}:
Let $A$ be a rank two, positive semidefinite matrix with all its entries positive and Hadamard power rank $k.$
By Perron's theorem, we can find a vector $p=(p_1,\ldots,p_n)$ in $\mathbb{R}^n$ with all its components positive
and a vector $q=(q_1,\ldots,q_n)$ in $\mathbb{R}^n$ such that $A=\begin{bmatrix}p_ip_j+q_iq_j\end{bmatrix}.$
Let $x_i=q_i/p_i,$ $1\le i\le n,$
and take $X$ to be the matrix
$$X=\begin{bmatrix}1+x_ix_j\end{bmatrix}.$$
Clearly $A=DX D,$
where $D$ is the diagonal matrix with diagonal entries $p_1,\ldots,p_n.$
Hence the inertia of $A^{\circ r}$ is the same as that of $X^{\circ r}.$
It is easy to see that $(p_i,q_i)$ is a scalar multiple of some $(p_j,q_j)$ if and only if $x_i=x_j.$
Hence the number of distinct $x_i$'s is equal to the Hadamard power rank of $A.$
Finally, we get Theorem \ref{thm_main} from Theorem \ref{thm4}.
\qed

\section{Monotonicity}

We first study monotonicity of the Hadamard power functions for the matrices of the form \eqref{eq11}, and
then use it to prove the main theorem of this section.

Let $x_1,\ldots,x_n$ be distinct nonzero real numbers and let $X$ be
the corresponding $n\times n$ matrix given by \eqref{eq11}.
Take $x_{n+1}=0,$ and let $X_0$ be the $(n+1)\times (n+1)$ matrix given by \eqref{eq11} corresponding to the $n+1$ distinct numbers $x_1,\ldots,x_n,x_{n+1},$ i.e.,
\begin{equation}
X_0=\begin{bmatrix}1+x_1^2 & \cdots & 1+x_1x_n & 1\\
1+x_2x_1 & \cdots & 1+x_2x_n & 1\\
\vdots & \vdots\vdots\vdots & \vdots & \vdots\\
1 & \cdots & 1 & 1\end{bmatrix}.\label{eq31}
\end{equation}
\begin{thm}\label{thm31}
Let $x_1,\ldots,x_n$ be distinct nonzero real numbers,
and let $X$ be the corresponding $n\times n$ matrix given by \eqref{eq11}.
Then $X^{\circ r}\ge E$ if and only if either $r> n-1$ or $r=0,1,\ldots,n-1.$
\end{thm}

\begin{proof}
Let $X_0$ be the $(n+1)\times (n+1)$ matrix given by \eqref{eq31} corresponding to $x_1,\ldots,x_n.$
Write $X_0^{\circ r}$ in the block form
$$X_0^{\circ r}=\begin{bmatrix}X^{\circ r} & e^T\\
e & \alpha\end{bmatrix},$$
where $e$ is the $n$-vector $(1,\ldots,1)$ and $\alpha=1.$
We see that $X^{\circ r}-E$ is the Schur complement of $\alpha$ in $X_0^{\circ r}.$
By Problem 4.5.P21 of \cite{hj}, $X_0^{\circ r}$ is positive semidefinite if and only if both $X^{\circ r}$ and $X^{\circ r}-E$ are positive semidefinite.
Since $x_i$'s are all distinct and nonzero,
by Theorem \ref{thm3}, we know that $X_0^{\circ r}$ is positive semidefinite if and only if $r>n-1$ or $r=0,1,\ldots,n-1.$
This implies that $X^{\circ r}-E$ is positive semidefinite if $r>n-1$ or $r=0,1,\ldots,n-1.$
If $r$ is not an integer and $0<r<n-1,$ then
at least one of $X^{\circ r}$ or $X^{\circ r}-E$ is not positive semidefinite.
But $X^{\circ r}\ge X^{\circ r}-E.$
Hence we obtain $X^{\circ r}-E$ is not positive semidefinite if $r$ is not an integer and $0<r<n-1.$
\end{proof}

\begin{thm}\label{thm32}
Let $x_1,\ldots,x_k$ be distinct nonzero numbers,
let $n_1,\ldots,n_k$ be positive integers, and let $n_{k+1}$ be a nonnegative integer with $n_1+\cdots+n_k+n_{k+1}=n.$
Let $X$ be the $n\times n$ matrix given by \eqref{eq11} corresponding to the $n$ numbers
${\underset{n_1\, \textrm{times}}{\underbrace{x_1,\ldots,x_1}}},\ldots,{\underset{n_k\, \textrm{times}}{\underbrace{x_k,\ldots,x_k}}},{\underset{n_{k+1}\, \textrm{times}}{\underbrace{0,\ldots,0}}}.$
If $X$ is entrywise positive, then	 $X^{\circ r}\ge E$ if and only if either $r>k-1$ or $r=0,1,\ldots,k-1.$
\end{thm}

\begin{proof}
Let $\wt{X}$ be the $k\times k$ matrix given by \eqref{eq11} corresponding to the distinct numbers $x_1,\ldots,x_k.$
Suppose $E_k$ denotes the $k\times k$ matrix with all entries one.
By Theorem \ref{thm31}, we know that
$\wt{X}^{\circ r}\ge E_k$ if and only if either $r>k-1$ or $r=0,1,\ldots,k-1.$
We can prove the theorem by using the facts that
$\wt{X}^{\circ r}-E_k$ is a $k\times k$ principal submatrix of $X^{\circ r}-E$ and $X^{\circ r}-E$ has rank less than or equal to $k.$
\end{proof}

Finally we give the last theorem of the paper.

\begin{thm}\label{thm_main2}
Let $A$ be an $n\times n$ rank two, positive semidefinite matrix,
and let $B$ be an $n\times n$ rank one, positive semidefinite matrix
such that both $A$ and $B$ have all their entries positive.
Suppose $A\ge B,$ and $A-B$ has rank one.
Let $\Hrk(A)=k.$
If $A-B$ has all diagonal entries nonzero, then $A^{\circ r}\ge B^{\circ r}$ if and only if either $r>k-1$ or $r=0,1,\ldots,k-1$;
and if $A-B$ has a zero diagonal entry, then $A^{\circ r}\ge B^{\circ r}$ if and only if either $r>k-2$ or $r=0,1,\ldots,k-2.$\end{thm}

\begin{proof}
Let $B=\begin{bmatrix}u_iu_j\end{bmatrix}$ and $A-B=\begin{bmatrix}v_iv_j\end{bmatrix}.$
Since $B$ is entrywise positive, all $u_i$'s are positive.
Take $x_i=v_i/u_i,$ $1\le i\le n,$ and take $X$ to be the $n\times n$ matrix $\begin{bmatrix}1+x_ix_j\end{bmatrix}.$
Since $A$ and $B$ are entrywise positive, so is $X.$
Let $D$ be the diagonal matrix $\diag(u_1^r,\ldots,u_n^r).$
Then 
\begin{equation}
A^{\circ r}=DX^{\circ r}D\textrm{ and }B^{\circ r}=DED.\label{eq41}
\end{equation}
Since $\Hrk(A)=k,$ exactly $k$ of the numbers $x_1,\ldots,x_n$ are distinct.
If $A-B$ has no zero diagonal entry, all these $k$ numbers are nonzero,
and if $A-B$ has a zero diagonal entry, then $k-1$ of these are nonzero.
We hence obtain the theorem by using Theorem \ref{thm32} and \eqref{eq41}.
\end{proof}

The Theorem \ref{thm_main2} is not true if we drop the condition that $A-B$ has rank one.
Let $x_1,\ldots,x_n$ be distinct positive numbers and let $A$ be the rank two, positive semidefinite matrix
$\begin{bmatrix}1+x_ix_j\end{bmatrix}.$
Let $n-2<r<n-1.$
We know that $A^{\circ r}$ is positive definite.
Hence we can find an $\alpha\in (0,1)$ such that
$A^{\circ r}\ge \alpha E.$
Take $\beta=\alpha^{1/r},$ and let $B$ be the $n\times n$ matrix $\beta E.$
Then we see that $A\ge B,$ $A-B$ has rank two,
and $A^{\circ r}\ge B^{\circ r}.$

{\bf Acknowledgements}:
The author thanks Professor Roger A. Horn and the two anonymous referees for their valuable suggestions that improved the readability of the paper. The author especially thanks one of the referees to suggest the use of Schur complements in the proof of Theorem \ref{thm31}, that led to much simplification of the proof.
Financial support from SERB
MATRICS grant number MTR/2018/000554 is also acknowledged.
\vskip.2in
 \section*{Conflict of interest}

 The author declares that she has no conflict of interest.

\end{document}